\documentclass[11pt,oneside,leqno]{amsart}
\usepackage[centertags]{amsmath}
\usepackage{amsfonts}
\usepackage{amsmath,amssymb}
\usepackage{latexsym}
\usepackage{color}

\newtheorem{theorem}{Theorem}[section]
\newtheorem{proposition}[theorem]{Proposition}
\newtheorem{corollary}[theorem]{Corollary}

\theoremstyle{definition}

\renewcommand{\theequation }{\thesection .\arabic{theequation}}

\begin{document}
\title[]{Ricci solitons in three-dimensional \\ paracontact geometry}
\author{Giovanni Calvaruso and Antonella Perrone}
\date{}

\address{Dipartimento di Matematica e Fisica \lq\lq E. De Giorgi\rq\rq \\
Universit\`a del Salento \\ Prov. Lecce-Arnesano \\
73100,  Lecce \\ Italy.}
\email{giovanni.calvaruso@unisalento.it; antonella\_perrone@unisalento.it}


\date{}

    \renewcommand{\theequation}{\thesection .\arabic{equation}}

\begin{abstract}
We completely describe paracontact metric three-manifolds whose Reeb vector field satisfies the Ricci soliton equation. While contact Riemannian (or Lorentz\-ian) Ricci solitons are necessarily trivial, that is, $K$-contact and Einstein, the paracontact metric case allows nontrivial examples. Both homogeneous and inhomogeneous nontrivial three-dimensional examples are explicitly described. Finally, we correct the main result of \cite{BeCr}, concerning three-dimensional normal paracontact Ricci solitons.
\end{abstract}

\maketitle

\bigskip\noindent
\section{Introduction}

A \emph{Ricci soliton} is a pseudo-Riemannian manifold $(M,g)$, admitting a smooth vector field $X$, such that
\begin{equation}\label{solit}
\mathcal{L}_X g+\varrho=\lambda g,
\end{equation}
where $\mathcal{L}_X$, $\varrho$ and $\lambda$ denote the Lie derivative in the direction of
$X$, the Ricci tensor and a real number, respectively. Sometimes, slightly different conventions are used. For example, in \cite{BeCr}, the Ricci soliton equation reads \lq\lq$\mathcal{L}_V g+2\varrho+2\lambda g=0$\rq\rq, which corresponds to \eqref{solit} taking $V=2X$ and changing sign to the constant $\lambda$. Clearly, Einstein manifolds satisfy the above equation, so that they are considered as trivial Ricci solitons. 
Referring to the above Equation~\eqref{solit}, a Ricci soliton is said to be \emph{shrinking}, \emph{steady} or
\emph{expanding}, according to whether $\lambda>0$, $\lambda=0$ or $\lambda<0$,
respectively.


Ricci solitons were introduced in Riemannian Geometry \cite{Ham} as the self-similar solutions of the \emph{Ricci flow}, and play an important role in understanding its singularities. A wide survey on Riemannian Ricci solitons may be found in \cite{Cao}. Recently, Ricci solitons  have also been extensively studied in pseudo-Riemannian settings. Among the several reasons for the growing interest of Theoretical Physicists toward Ricci solitons, we cite their relation with String Theory, and the fact that Equation~\eqref{solit} is a special case of the Einstein field equations. For some recent results and further references on pseudo-Riemannian Ricci solitons, we may refer to \cite{isr},\cite{CF},\cite{CZ},\cite{PT} and references therein. 

Given a class of pseudo-Riemannian manifolds $(M,g)$, it is then a natural problem to solve Equation~\eqref{solit}, especially when it holds for a smooth vector field playing a special role in the geometry of the manifold itself.  Under this point of view, the Reeb vector field $\xi$ of a contact metric manifold would be a natural candidate. However, in these settings, a strong rigidity result holds: the Reeb vector field of a contact Riemannian (or Lorentzian) manifold $(M,\varphi,\xi,\eta,g)$ satisfies \eqref{solit} if and only if $(M,\varphi,\xi,\eta,g)$ is $K$-contact Einstein \cite{CP2}. Thus, contact Riemannian or Lorentzian Ricci solitons are necessarily trivial.
   
Paracontact metric structures were introduced in \cite{KW}, as a natural odd-dimensional counterpart to paraHermitian structures, just like contact metric structures correspond to the Hermitian ones. Since the basic paper \cite{Z}, in the last years several authors studied paracontact metric structures and their further generalization, namely, almost paracontact metric structures, emphasizing similarities and differences with respect to the most investigated contact case. Some recent results on paracontact and almost paracontact metric structures may be found in  \cite{BeCr},\cite{CIll},\cite{CM},\cite{CPA},\cite{CKM},\cite{DO},\cite{IVZ},\cite{W3},\cite{W4} and references therein. Using the relationship between Ricci solitons and infinitesimal harmonic transformations pointed out in \cite{SS}, it was proved in \cite{CPe} that a {\em paracontact Ricci soliton}, that is, a paracontact metric manifold for which the Reeb vector field $\xi$ satisfies \eqref{solit}, is necessarily expanding and must satisfy $Q\xi =-2n \xi$, where $Q$ denotes the Ricci operator.
 
In this paper, we answer positively the open question whether there exist nontrivial paracontact Ricci solitons, obtaining a complete description of three-dimensional paracontact Ricci solitons. After reporting in Section~2 some basic information about almost paracontact metric structures and infinitesimal harmonic transformations, in Section~3  we investigate three-dimensional paracontact Ricci solitons, emphasizing their relationship with {\em paracontact $(\kappa,\mu)$-spaces}.  Some byproducts of this description of three-dimensional paracontact Ricci solitons are relevant on their own. In particular, we completely describe paracontact metric three-manifolds whose Reeb vector field $\xi$ is an infinitesimal harmonic transformation, and we prove that if $\xi$ is affine Killing, then it is necessarily Killing. Homogeneous and inhomogeneous examples of three-dimensional paracontact Ricci solitons are explicitly described in Section~4, also providing the description in local Darboux coordinates of an arbitrary three-dimensional paracontact metric structure. These examples also show that, differently from the {\em contact} metric case \cite{Bo}, a paracontact $(\kappa,\mu)$-space needs not to be locally isometric to some Lie group. Finally, in Section~5 we shall correct the main result of \cite{BeCr} on three-dimensional normal almost paracontact Ricci solitons (namely, Theorem~3.1 in \cite{BeCr}), showing that in general their Reeb vector field needs not to be conformal Killing, and that in such a case, a three-dimensional normal almost paracontact Ricci soliton is indeed trivial, being a manifold of constant sectional curvature.

\bigskip\noindent
\section{Preliminaries}
\setcounter{equation}{0}

\subsection{Three-dimensional almost paracontact metric structures}
An {\em almost paracontact structure} on a $(2n+1)$-dimensional (connected) smooth manifold $M$ is given by a triple $(\varphi,\xi,\eta)$, where $\varphi$ is a $(1,1)$-tensor, $\xi$ a global vector field and $\eta$ a $1$-form, such that
 \begin{equation}\label{almostpc}
\varphi (\xi)=0, \qquad \eta \circ \varphi =0, \qquad \eta (\xi)=1,  \qquad \varphi ^2 = Id -\eta \otimes \xi
\end{equation}
and the restriction $J$ of $\varphi$ on the horizontal distribution ker$\eta$ is an almost paracomplex structure (that is, the eigensubbundles $D^+, D^-$ corresponding to the eigenvalues $1,-1$ of $J$ have equal dimension $n$).

A pseudo-Riemannian metric $g$ on $M$ is {\em compatible} with the almost paracontact structure $(\varphi,\xi,\eta)$ if
\begin{equation}\label{compg}
g(\varphi X, \varphi Y) =-g(X,Y) +\eta(X)\eta(Y).
\end{equation}
In such a case, $(\varphi,\xi,\eta,g)$ is said to be an {\em almost paracontact metric structure}.
Observe that, by \eqref{almostpc} and \eqref{compg}, $\eta(X)=g(\xi,X)$ for any compatible metric. Any almost paracontact structure admits compatible metrics, which, because of \eqref{compg}, have signature $(n+1,n)$. The {\em fundamental $2$-form} $\Phi$ of an almost paracontact metric structure $(\varphi,\xi,\eta,g)$ is defined by $\Phi(X,Y)=g(X,\varphi Y)$, for all tangent vector fields $X,Y$. If $\Phi=d\eta$, 
then the manifold $(M,\eta,g)$ (or $(M,\varphi,\xi,\eta,g)$) is called a {\em paracontact metric manifold} and $g$ the {\em associated metric}.

\noindent

Throughout the paper, we shall denote with $\nabla$ the Levi-Civita connection and by $R$ the curvature tensor of 
$g$, taken with the sign convention
$$
	R(X,Y)=\nabla_{[X,Y]}-[\nabla_X,\nabla_Y].
$$
An almost paracontact metric structure $(\varphi,\xi,\eta,g)$ is said to be {\em normal} if 
\begin{equation}\label{normal}
[\varphi,\varphi ] -2d\eta\otimes\xi=0 ,
\end{equation} 
where $[\varphi,\varphi]$ is the Nijenhuis torsion tensor of $\varphi$.  In dimension three, normal almost paracontact metric structures are characterized by condition
\begin{equation}\label{trans}
(\nabla _X \varphi)Y=\bar \alpha (g(\varphi X,Y)\xi-\eta(Y) \varphi X)+\bar \beta (g(X,Y)\xi-\eta(Y)X),
\end{equation} 
(see \cite{W3}), for two smooth functions $\bar\alpha, \bar\beta$. 
When \eqref{trans} holds, an almost paracontact metric structure is said to be  (see \cite{W4})
\begin{itemize}
\vspace{2pt}\item[(a)] quasi-para-Sasakian if $\bar \alpha=0\neq \bar \beta $ (when $\bar\beta$ is a constant, a quasi-para-Sasakian structure is called $\beta$-para-Sasakian);
\item[(b)] $\alpha$-para-Kenmotsu if $\bar \alpha \neq 0$ is a constant and $\bar \beta=0$;
\item[(c)] {\em  para-cosymplectic} if it is normal with $\bar \alpha=\bar \beta=0$.
\end{itemize}

\noindent
Generalizing the paracontact metric case, for any almost paracontact metric manifold $(M,\varphi,\xi,\eta,g)$, one can introduce the $(1,1)$-tensor $h$, defined by $hX=\frac{1}{2}(\mathcal L_{\xi} \varphi)X$. Since $\nabla$ is torsionless, one has 
\begin{eqnarray}  \label{h}
hX=\frac{1}{2} ((\nabla_\xi \varphi)X-\nabla_{\varphi X}\xi+\varphi(\nabla_X \xi)).	
\end{eqnarray}
As proved in \cite{CPA}, if $\dim M=3$, then $(M,\varphi,\xi,\eta,g)$ is normal if and only if $h=0$. 
Note that condition $h=0$ is equivalent to $\xi$ being Killing for paracontact metric manifolds, but not for almost paracontact ones (see also Proposition~\ref{p5} below). We also recall the following result.

\begin{proposition}{\bf \cite{CPA}}\label{first}
Let $(M,\varphi,\xi,\eta,g)$ be a $(2n+1)$-dimensional almost paracontact metric manifold. Then, the following properties are equivalent:

\begin{itemize}
\item[(i)]$\xi \in {\rm ker}d\eta$;
\vspace{2pt}\item[(ii)]  $\mathcal L_{\xi} \eta=0$;
\vspace{2pt}\item[(iii)] $\nabla_\xi \xi=0$;
\vspace{2pt}\item[(iv)] $[\xi,X]\in {\rm ker}\eta$, for any $X \in {\rm ker}\eta$. 
\end{itemize}
\end{proposition}

\noindent
Conditions listed in the above Proposition~\ref{first} hold for both paracontact metric and normal almost paracontact metric structures. 

\noindent
Finally, we recall that any almost paracontact metric manifold $(M^{2n+1},\varphi,\xi,\eta,g)$ admits (at least, locally) a {\em $\varphi$-basis} \cite{Z}, that is, a pseudo-orthonormal basis of vector fields of the form $\{\xi, E_1,..,E_n,\varphi E_1, ..,\varphi E_n\},$ where $\xi,E_1,..,E_n$ are space-like vector fields and so, by \eqref{compg}, vector fields $\varphi E_1, ..,\varphi E_n$ are time-like. Observe that when $\dim M=3$, any (local) pseudo-orthonormal basis of ker$\eta$ determines a $\varphi$-basis, up to sign. In fact, if $\{e_2,e_3\}$ is a (local) pseudo-orthonormal basis of ker$\eta$, with $e_3$, time-like, then by \eqref{compg} we have that $\varphi e_2 \in$ker$\eta$ is time-like and orthogonal to $e_2$. So, $\varphi e_2=\pm e_3$ and $\{\xi,e_2,\pm e_3\}$ is a $\varphi$-basis.

\subsection{Infinitesimal harmonic transformations and Ricci solitons}

Denote by $(M,g)$ a pseudo-Riemannian manifold, by $\nabla$ its the Levi-Civita connection and by $f : x \mapsto x'$ a point transformation in $(M, g)$. If $\nabla '(x)$ is obtained bringing back $\nabla (x')$ to $x$ by $f^{-1}$, the {\em Lie difference} at $x$ is defined as $\nabla'(x)-\nabla(x)$. The map $f$ is said to be {\em harmonic} if ${\rm tr}(\nabla'(x)-\nabla(x))=0$.

Given a vector field $V$ on $M$, the Lie derivative $L_{V} \nabla$ at $x$ then corresponds to $\nabla'(x)-\nabla(x)$, where $\nabla'(x)=f^* _t (\nabla(x'))$. Vector field $V$ is said to be an {\em infinitesimal harmonic transformation} if it generates a group of harmonic transformations, that is, 
$${\rm tr}(\mathcal L _V \nabla)=0.$$ 
Clearly, infinitesimal harmonic transformations generalize {\em affine Killing vector fields}, for which $\mathcal L _V \nabla=0$. A vector field $V$ determining a Ricci soliton (that is, satisfying \eqref{solit}) is necessarily an infinitesimal harmonic transformation \cite{SS}. The same argument, initially obtained in the Riemannian case, also applies to pseudo-Riemannian manifolds. We may refer to \cite{SS} and references therein for more details about infinitesimal harmonic transformations.

\section{$3D$ paracontact Ricci solitons}
\setcounter{equation}{0} 

We first provide a local description of all almost paracontact metric three-manifolds (called \emph{natural} in \cite{CPA}) satisfying any of conditions listed in the above Proposition~\ref{first}. Let $(M,\varphi,\xi,\eta,g)$ be such a manifold. Then, because of Proposition~\ref{first}, equation~(\ref{h}) reduces to
\begin{eqnarray} \label{1}
	2hX=\varphi(\nabla_X \xi)-\nabla_{\varphi X}\xi.
\end{eqnarray}
Let now $\{\xi,e,\varphi e\}$ denote a (local) $\varphi$-basis on $M$, with $\varphi e$ time-like. Then, 
\begin{equation}\label{acca}
he=a_1e+a_2\varphi e, \qquad 	 h\varphi e=(-\varphi he)= -a_2e-a_1\varphi e,
\end{equation}
for some smooth functions  $a_1,a_2$. Consequently, 
\begin{eqnarray} \label{h^2}
	\left\|h\right\|^2={\rm tr} h^2=2(a_1^2-a_2^2).
\end{eqnarray}
In particular, by \eqref{acca} and \eqref{h^2} we have that the following conditions are equivalent:
\begin{itemize}
\item[(i)] $h^2=0$, that is, $h$ is two-step nilpotent;
\vspace{2pt}\item[(ii)] tr$h^2=0$;
\vspace{2pt}\item[(iii)] $a_2= \varepsilon a_1=\pm a_1$. 
\end{itemize}
Since $\nabla _e \xi$ is orthogonal to $\xi$, there exist two smooth functions $b_1,b_2$, such that $\nabla_e \xi= b_1e+b_2\varphi e$. So, (\ref{1}) yields $\nabla_{\varphi e} \xi= (b_2-2a_1)e+(b_1-2a_2)\varphi e$. Moreover,
$$\nabla_\xi e=g(\nabla_\xi e,\xi)\xi+g(\nabla_\xi e,e)e-g(\nabla_\xi e,\varphi e)\varphi e=-g(\nabla_\xi e,\varphi e)\varphi e=a_3\varphi e,
$$
where we put $a_3:=g(\nabla_\xi \varphi e,e)$. By similar computations and taking into account the compatibility of $g$, we obtain
\begin{eqnarray} \label{nabla}
\quad \quad   
	\left\{
	\begin{array} 
	{ll}
	   \nabla_e \xi= b_1e+b_2\varphi e, &\quad \nabla_{\varphi e}\xi=(b_2-2a_1)e+(b_1-2a_2){\varphi e}, \\[4pt]
	   \nabla_\xi e=a_3{\varphi e}, \  &\quad \nabla_\xi {\varphi e}=a_3 e, \\[4pt]
		 \nabla_e e=-b_1\xi+a_4{\varphi e}, \ 	&\quad \nabla_{\varphi e}{\varphi e}=(b_1-2a_2)\xi+a_5e, \\[4pt]
		 \nabla_e {\varphi e}=b_2\xi+a_4e, \  &\quad \nabla_{\varphi e} e=(2a_1-b_2)\xi+a_5{\varphi e}, 
	\end{array}
	\right.
\end{eqnarray}

\noindent
for some real smooth functions $a_i,b_j$. Equivalently, the Lie brackets of $\xi,e,\varphi e$ are described by
\begin{equation} \label{2}
\left\{
 \begin{array}{l}
    [\xi,e]=-b_1e+(a_3-b_2)\varphi e,  \\[2pt]
     [\xi,\varphi e]=(a_3+2a_1-b_2)e+(2a_2-b_1)\varphi e, \\[2pt]
     [e,\varphi e]=2(b_2-a_1)\xi+a_4e-a_5\varphi e.
 \end{array}	
\right.
\end{equation}
and must satisfy the Jacoby identity, which, by standard calculations, is proved to be equivalent to the following system equations:
\begin{equation} \label{3}
\left\{ \begin{array}{l}
				\xi(b_2-a_1)-2(a_2-b_1)(b_2-a_1)=0, \\[2pt]
				\xi(a_4)-e(a_3+2a_1-b_2)-\varphi e(b_1)-a_4(2a_2-b_1)-a_5(a_3+2a_1-b_2)=0,	\\[2pt]
				\xi(a_5)+e(2a_2-b_1)+\varphi e(b_2-a_3)+a_4(b_2-a_3)+a_5b_1=0.
\end{array}
\right.
\end{equation}
We then have the following result.

\begin{proposition}\label{p5}
Any three-dimensional almost paracontact metric three-manifold, satisfying any of conditions listed in 
Proposition~{\em\ref{first}}, is locally described by Equations~\eqref{2}-\eqref{3} with respect to a local $\varphi$-basis $\{\xi, e, \varphi e\}$, for some smooth functions $a_i,b_j$. In particular: 
\begin{enumerate}
\item 
 $\eta$ is a contact form if and only if $a_1-b_2\neq 0$, equivalently, tr$(\varphi \nabla \xi)\neq0$;
\vspace{2pt}\item $(\varphi,\xi,\eta,g)$ is a paracontact metric structure if and only if $a_1-b_2=1$; 
\vspace{2pt}\item $h=0$ if and only if $a_1=a_2=0$; 
\vspace{2pt}\item $\xi$ is a Killing vector field if and only if $h=0$ and $b_1=0$; 
\vspace{2pt}\item div$\xi=0$ if and only if $a_2=b_1$.
\end{enumerate}
\end{proposition}
\begin{proof}
From (\ref{2}) we get
$$\eta \wedge d\eta(\xi,e,\varphi e)=d\eta(e,\varphi e)=-\frac 12 \eta[e,\varphi e]=a_1-b_2$$
and
\begin{equation}\label{trphi}
{\rm tr}(\varphi \nabla \xi)=2(b_2-a_1).
\end{equation}
Thus, $\eta \wedge d\eta(\xi,e,\varphi e)\neq0$ if and only if $a_1-b_2\neq0$, which proves (1).

\noindent
Now, as $d\eta(\xi,\varphi e)=d\eta(\xi,e)=0=g(\xi,e)=g(\xi,\varphi e)$ and $g(e,\varphi(\varphi e))=1$, we have that $d\eta=g(\cdot,\varphi \cdot)$ if and only if $a_1-b_2=1$, which proves (2).

\noindent
Next, (3) follows at once from \eqref{acca} and $h\xi=0$. To prove (4), we compute $(\mathcal L_{\xi} g)(X,Y)=g(\nabla_X \xi,Y)+g(\nabla_Y \xi,X)$. Also taking into account  $\nabla_\xi \xi =0$, we find $(\mathcal L_{\xi} g)(\xi,\cdot)=0$. Moreover, by (\ref{nabla}), we get 
$$\begin{array}{l} (\mathcal L_{\xi} g)(e,e)=2g(\nabla_e \xi,e)=2b_1, \\[4pt]
(\mathcal L_{\xi} g)(\varphi e,\varphi e)=2g(\nabla_{\varphi e} \xi,\varphi e)=2(2a_2-b_1),\\[4pt]
(\mathcal L_{\xi} g)(e,\varphi e)=g(\nabla_{e} \xi,\varphi e)+g(\nabla_{\varphi e} \xi,e)=-2a_1.
\end{array}$$
Therefore, $\xi$ is a Killing vector field if and only if $a_1=a_2=b_1=0$, that is, $h=0$ and $b_1=0$.
Finally,
\begin{equation}\label{divxi}
{\rm div} \xi=g(\nabla_\xi \xi,\xi)+g(\nabla_e \xi,e)-g(\nabla_{\varphi e} \xi,\varphi e)=2(b_1-a_2),
\end{equation}
which proves (5). 
\end{proof}

We now focus on the paracontact metric case. By the above Proposition~\ref{p5}, if $(\varphi,\xi,\eta,g)$ is a paracontact metric structure, then $b_2=a_1-1$, which, by the first equation in \eqref{3}, also yields $b_1=a_2$. Therefore, any three-dimensional paracontact metric three-manifold is (locally) described by equations
\begin{equation} \label{2pc}
\left\{
 \begin{array}{l}
    [\xi,e]=-a_2e+(a_3-a_1+1)\varphi e,  \\[2pt]
     [\xi,\varphi e]=(a_3+a_1+1)e+a_2\varphi e, \\[2pt]
     [e,\varphi e]=-2\xi+a_4e-a_5\varphi e,
 \end{array}	
\right.
\end{equation}
for some smooth functions $a_1,\dots,a_5$, satisfying 
\begin{eqnarray} \label{3pc}
\left\{ \begin{array}{l}
				\xi(a_4)-e(a_3+a_1)-\varphi e(a_2)-a_4 a_2-a_5(a_3+a_1+1)=0,	\\[2pt]
				\xi(a_5)+e(a_2)+\varphi e(a_1-a_3)+a_4(a_1-a_3-1)+a_5 a_2=0.
\end{array}
\right.
\end{eqnarray}
Consequently, the Levi-Civita connection is described by
\begin{eqnarray} \label{nablapc}
\quad \quad   
	\begin{array} 
	{lll}
\quad	  \nabla_{\xi} \xi=0 & \quad \nabla_e \xi= a_2e+(a_1-1)\varphi e, &\quad \nabla_{\varphi e}\xi=-(a_1+1)e-a_2{\varphi e}, \\[4pt]
\quad \nabla_\xi e=a_3{\varphi e}, & \quad \nabla_e e=-a_2\xi+a_4{\varphi e}, &\quad  
\nabla_{\varphi e} e=(a_1+1)\xi+a_5{\varphi e} \\[4pt]
\quad		\nabla_\xi {\varphi e}=a_3 e, &  \quad\nabla_e {\varphi e}=(a_1-1)\xi+a_4e, 	&\quad \nabla_{\varphi e}{\varphi e}=-a_2\xi+a_5e,
	\end{array}
\end{eqnarray}

\noindent
In order to describe three-dimensional paracontact metric structures, whose Reeb vector field $\xi$ is an infinitesimal {harmonic transformation,} we now use the above equations to calculate tr$\mathcal L_{\xi} \nabla$. Explicitly, {since $(\mathcal L_Z\nabla)(X,Y)=\mathcal L_Z \nabla_X Y-\nabla_{\mathcal L_Z X}Y-\nabla_{X} \mathcal L_Z Y$}, we get
{
\begin{align*}
(\mathcal L_{\xi} \nabla)(\xi,\xi)=&0, \\[2pt]	   
(\mathcal L_{\xi} \nabla)(e,e)=&-\big(\xi(a_2)+2a_2^2+2a_1(a_3-a_1+1)\big)\xi+\big(e(a_2)+2a_1a_4\big)e \\&+\big(\xi(a_4)-e(a_3-a_1)+3a_2a_4-a_5(a_3-a_1+1)\big) \varphi e, \\[2pt]	   
(\mathcal L_{\xi} \nabla)(\varphi e ,\varphi e)=& -\big(\xi(a_2)-2a_2^2+2a_1(a_3+a_1+1)\big)\xi-\big(\varphi e(a_2)+2a_1a_5\big)\varphi e \\&+\big(\xi(a_5)-\varphi e(a_3+a_1)-3a_2a_5-a_4(a_3+a_1+1)\big)e
\end{align*}
}
and so,
\begin{align*}
{\rm tr}(\mathcal L_{\xi} \nabla)=& (\mathcal L_{\xi} \nabla)(\xi,\xi)+(\mathcal L_{\xi} \nabla)(e ,e)-(\mathcal L_{\xi} \nabla)(\varphi e ,\varphi e) \\[2pt]	   
=& {4}\big(a_1^2-a_2^2\big)\xi \\
+&\big(e(a_2)+\varphi e(a_3+a_1)-\xi(a_5)+2a_1a_4 +3a_2a_5+a_4(a_3+a_1+1)\big)e 
\\
+&\big(e(a_1-a_3)+\varphi e(a_2)+\xi(a_4)+2a_1a_5 +{3a_2a_4- }a_5(a_3-a_1+1)\big) \varphi e.
\end{align*}
Substituting $\xi(a_4),\xi(a_5)$ from \eqref{3pc}, it is easily seen that ${\rm tr}(\mathcal L_{\xi} \nabla)=0$ if and only if $a_2=\varepsilon a_1=\pm a_1$ and $(e+\varepsilon \varphi e)(a_1)+2\varepsilon a_1a_4+2a_1a_5=0$. So, we proved the following result.

\begin{proposition}\label{xiIAT}
Let $(M,\varphi,\xi,\eta,g)$ be a three-dimensional paracontact metric manifold. Then, the Reeb vector field $\xi$ is an infinitesimal harmonic transformation if and only if the manifold  is locally described by 
\begin{equation} \label{2h2}
\left\{
 \begin{array}{l}
    [\xi,e]=-\varepsilon a_1 e+(a_3-a_1+1)\varphi e,  \\[2pt]
     [\xi,\varphi e]=(a_3+a_1+1)e+\varepsilon a_1\varphi e, \\[2pt]
     [e,\varphi e]=-2\xi+a_4e-a_5\varphi e,
 \end{array}	
\right.
\end{equation}
with respect to a local $\varphi$-basis $\{\xi, e, \varphi e\}$, for some smooth functions $a_1,a_3,a_4,a_5$, satisfying 
\begin{eqnarray} \label{3h2}
\left\{ \begin{array}{l}
				\xi(a_4)-e(a_3) +\varepsilon a_1 a_4-a_5(a_3-a_1+1)=0,	\\[2pt]
				\xi(a_5)-\varphi e(a_3)-\varepsilon a_1 a_5 -a_4(a_3+a_1+1)=0, \\[2pt]
	(e+\varepsilon \varphi e)(a_1)+2\varepsilon a_1a_4+2a_1a_5=0.			
\end{array}
\right.
\end{eqnarray}
\end{proposition}

\noindent
We observe that by the above Proposition, if $\xi$ is an infinitesimal harmonic transformation, then $a_2=\varepsilon a_1$ and so, tr$h^2=0$, compatibly with the result proved in \cite{CP2} for paracontact metric manifolds of arbitrary dimension, whose Reeb vector field is an infinitesimal harmonic transformation.

Before describing three-dimensional paracontact Ricci solitons, we now consider the case when $\xi$ is an affine Killing vector field. In such a case, from \eqref{nablapc} we get 
\begin{eqnarray*}
	(\mathcal L_{\xi} \nabla)(\xi , e)=\big(\xi(a_2)+2a_1a_3\big)e+\big(\xi(a_1) + 2a_2a_3\big)\varphi e.
\end{eqnarray*}
Hence, $(\mathcal L_{\xi} \nabla)(\xi, e)=0$, $(\mathcal L_{\xi} \nabla)(e , e)=0$ and $(\mathcal L_{\xi} \nabla)(\varphi e ,\varphi e)=0$ respectively give 
\begin{eqnarray} 
\left\{ \begin{array}{l}
\xi(a_2)+2a_1a_3=0, \\[2pt]	
\xi(a_2)+2a_2^2+2a_1(a_3-a_1+1)=0, \\[2pt]	   
\xi(a_2)-2a_2^2+2a_1(a_3+a_1+1)=0,
\end{array}
\right.
\end{eqnarray}
%
%
%
which easily yield $a_1=a_2=0$, that is, by \eqref{acca}, $h=0$. So, $\xi$ is Killing and, taking into account Theorem~2.2 in \cite{CIll}, we proved the following result.

\begin{theorem}
If the Reeb vector field $\xi$ of a three-dimensional paracontact metric manifold $(M,\varphi,\xi,\eta,g)$ is affine Killing, then $\xi$ is Killing (and so, $M$ is paraSasakian).
\end{theorem}

We now determine the Ricci tensor of any paracontact metric three-manifold whose Reeb vector field is an infinitesimal harmonic transformation. Using \eqref{nablapc} with $a_2=\varepsilon a_1$ and taking into account \eqref{3h2}, standard calculations yield
\begin{equation}\label{curvpc}
\left\{
\begin{array}{l}
R(\xi,e)\xi= -\big(\varepsilon \xi(a_1)+2a_1a_3{+}1 \big)e -\big(\xi(a_1)+2\varepsilon a_1a_3 \big) \varphi e, \\[2pt]	
R(\xi, \varphi e)\xi= \big(\xi(a_1)+2{\varepsilon}a_1a_3 \big)e +\big(\varepsilon \xi(a_1)+{2a_1a_3}-1 \big) \varphi e, \\[2pt]
R(e, \varphi e)\xi= 0, \\[2pt]
R(e, \varphi e)e= \big(\varphi e (a_4)-e(a_5) +1-2a_3+a_4^2-a_5^2\big) \varphi e,  	
\end{array}
\right.
\end{equation}
which easily imply that with respect to $\{\xi, e, \varphi e \}$, the Ricci tensor $\varrho$ is completely described by
\begin{equation}\label{rhopc}
\varrho =\left(\begin{array}{ccc}
-2 & 0 & 0 \\[2pt]
0 & -B-\varepsilon A & A \\[2pt]
0 & A & B-\varepsilon A
\end{array}\right),  \quad \begin{array}{l} A:=\xi(a_1)+{2\varepsilon a_1a_3, } \\[4pt]
B:= e(a_5)-\varphi e (a_4)+2a_3-a^2_4+a_5^2. \end{array}
\end{equation}
Observe that by \eqref{rhopc}, we see that the Ricci operator $Q$ satisfies $Q\xi=-2\xi$, compatibly with the result of \cite{CP2} concerning the case when $\xi$ is an infinitesimal harmonic transformation.

Next, from \eqref{curvpc} we have that $R(\xi,e,\xi,e)=-\varepsilon A-1$. On the other hand,  it is well known that in dimension three, the curvature tensor $R$ satisfies
\begin{align} \label{curv3d}
	R(X,Y,Z,V)= & \ g(X,Z)\varrho(Y,V)-g(Y,Z)\varrho(X,V)+g(Y,V)\varrho(X,Z) \\
	&-g(X,V)\varrho(Y,Z)-\frac{r}{2}\big( g(X,Z)g(Y,V)-g(Y,Z)g(X,V) \big), \nonumber
\end{align}
where $r$ denotes the scalar curvature.  In particular, for $X=Z=\xi$ and $Y=V=e$, we then get $R(\xi,e,\xi,e)=\varrho(e,e)-2 - \frac{r}{2}$ and so,  $ \varrho(e,e)=- \varepsilon A + \frac{r}{2} +1$. Comparing with \eqref{rhopc}, we then find $B=-\frac{r}{2}-1$. Consequently, \eqref{rhopc} becomes
\begin{equation}\label{newrhopc}
\varrho =\left(\begin{array}{ccc}
-2  & 0           													& 0 \\[4pt]
0   & -\varepsilon A+  \frac{r}{2} +1 & A \\[4pt]
0   & A & -\varepsilon A - \frac{r}{2} -1
\end{array}\right),  
\end{equation}
%
%
Next, for any paracontact metric three-manifold $(M,\varphi,\xi,\eta,g)$, if $h^2=0$, then applying \eqref{nablapc} with $a_2=\varepsilon a_1$, we easily find that, with respect to $\{\xi, e, \varphi e \}$, 
\begin{equation}\label{liexi}
\mathcal L _{\xi} g =\left(\begin{array}{ccc}
0 & 0 & 0 \\[2pt]
0 & 2\varepsilon a_1 & -2a_1 \\[2pt]
0 & -2a_1 & 2\varepsilon a_1
\end{array}\right).
\end{equation}
Therefore, $\xi$ satisfies equation \eqref{solit} if and only if
\begin{equation}\label{RSpc}
\left\{
\begin{array}{l}
\lambda=-2,  \\
A=\xi(a_1)+2\varepsilon a_1a_3=2a_1, \\
r= - 6.
\end{array}
\right.
\end{equation}
%
%
%
Thus, we proved the following result.

\begin{theorem}\label{thRS}
A three-dimensional paracontact metric manifold $(M,\varphi,\xi,\eta,g)$ is a paracontact Ricci soliton if and only if the manifold  is locally described by equations \eqref{2h2},\eqref{3h2} and \eqref{RSpc}, with respect to a local $\varphi$-basis $\{\xi, e, \varphi e\}$, for some smooth functions $a_1,a_3,a_4,a_5$. In particular, the Ricci soliton is necessarily expanding.
\end{theorem}

\noindent
In the following section, we shall use the above Theorem~\ref{thRS} to describe explicitly some nontrivial three-dimensional paracontact Ricci solitons. 

{We now end this section clarifying the relationship between three-dimensional paracontact Ricci solitons and $(\kappa,\mu)$-spaces. By definition, a \emph{paracontact $(\kappa,\mu)$-space} is a paracontact metric manifold $(M,\varphi,\xi,\eta,g)$, satisfying the curvature condition
\begin{equation} \label{kappamu}
{R}(X,Y){\xi}=\kappa ({\eta}(X)Y-{\eta}(Y)X)+\mu({\eta}(X){h}Y-{\eta}(Y){h}X),
\end{equation}
for all vector fields $X,Y$ on $M$, where  $\kappa$ and $\mu$ are smooth functions. These manifolds generalize the paraSasakian ones, for which $\kappa=-1$ and $\mu$ is undetermined. We may refer to \cite{CKM} and references therein for recent results on paracontact $(\kappa,\mu)$-spaces, essentially focused on the case when $\kappa$ and $\mu$ are constant. {Observe that our curvature convention is opposite to the one used in \cite{CKM}}.

Because of its tensorial character, it suffices to check equation~\eqref{kappamu} on a $\varphi$-basis.
By Proposition~\ref{xiIAT}, when $\xi$ is an infinitesimal harmonic transformation, $(M,\varphi,\xi,\eta,g)$ is locally described by \eqref{2h2} and \eqref{3h2}.  Henceforth, from \eqref{curvpc} and the definition of $A$ in \eqref{rhopc}, we have
$$
R(\xi,e)\xi= -\big(\varepsilon A {+}1 \big)e -A \varphi e, \quad	
R(\xi, \varphi e)\xi= A e +\big(\varepsilon A-1 \big) \varphi e, \quad
R(e, \varphi e)\xi= 0,
$$
which, compared with \eqref{kappamu}, easily yields that $(M,\varphi,\xi,\eta,g)$ is a $(\kappa,\mu)$-space, with $\kappa=-1$ and $a_1 \mu=-\varepsilon A$.  

In particular, if $(M,\varphi,\xi,\eta,g)$ is a three-dimensional nontrivial paracontact Ricci soliton (Theorem~\ref{thRS}), then \eqref{RSpc} yields $\mu=-2\varepsilon$. Conversely, suppose that $(M,\varphi,\xi,\eta,g)$ is a three-dimensional paracontact $(\kappa,\mu)$-space, with $\kappa=-1$ and $\mu=-2\varepsilon$. Since $\kappa=-1$, tensor $h$ is two-step nilpotent (\cite{CKM}, Lemma~3.2). Therefore, $(M,\varphi,\xi,\eta,g)$ is locally described by \eqref{2pc},\eqref{3pc} with $a_2=\varepsilon a_1$. In particular, $\mathcal L_{\xi} g$ is then given by \eqref{liexi}.

\noindent
By the $(\kappa,\mu)$-condition \eqref{kappamu} with  $\kappa=-1$ and $\mu=-2\varepsilon$, we have
$$\begin{array}{l}
R(\xi,e)\xi= -\big(2\varepsilon a_1+1\big)e -2a_1 \varphi e, \;
R(\xi, \varphi e)\xi= 2a_1 e +\big(2\varepsilon a_1-1 \big) \varphi e, \;
R(e, \varphi e)\xi= 0,
\end{array}$$
which easily yield 
\begin{equation}\label{rhokmu}
\varrho =\left(\begin{array}{ccc}
-2  & 0   & 0 \\[4pt]
0   & \frac{r}{2} +1-2\varepsilon a_1 & 2a_1 \\[4pt]
0   & 2a_1 & - \frac{r}{2} -1-2\varepsilon a_1
\end{array}\right).
\end{equation}
Replacing from \eqref{liexi} and \eqref{rhokmu} into \eqref{solit}, we then obtain the following result, which contains a complete characterization of three-dimensional paracontact Ricci solitons.}

\begin{theorem}\label{thkmu}
Let $(M,\varphi,\xi,\eta,g)$ denote a three-dimensional paracontact metric manifold. 

\smallskip
(a) If $(M,\varphi,\xi,\eta,g)$ is not paraSasakian and $\xi$ is an infinitesimal harmonic transformation, then $(M,\varphi,\xi,\eta,g)$ is a $(\kappa,\mu)$-space, with $\kappa=-1$ and $\mu=-\varepsilon A / a_1$. 

\smallskip
(b) In particular, $(M,\varphi,\xi,\eta,g)$ is a nontrivial paracontact Ricci soliton if and only if is a $(\kappa,\mu)$-space, with $\kappa=-1$ and $\mu=-2\varepsilon$, of scalar curvature $r=-6$.
\end{theorem}

The above Theorem~\ref{thkmu} answers the question stated in \cite{CPe}, concerning the existence of nontrivial paracontact Ricci solitons: in dimension three, they coincide with a remarkable class of paracontact $(\kappa,\mu)$-spaces. This result is a remarkable difference between the paracontact and the contact metric cases with regard to Ricci solitons. In fact, a {\em contact} metric Ricci soliton, either Riemannian or Lorentzian, is necessarily trivial, that is, Sasakian Einstein (see \cite{CP2} and references therein). Explicit examples of nontrivial paracontact Ricci solitons, both homogeneous and inhomogeneous, will be described in the next Section. 

We shall now specify the Segre type of the Ricci operator of a three-dimensional nontrivial paracontact Ricci soliton. Because of the symmetries of the curvature tensor, the Ricci tensor $\varrho$ is symmetric \cite{O'N}. Consequently, the {\em Ricci operator} $Q$, defined by $g(QX,Y)=\varrho(X,Y)$, is self-adjoint. In the Riemannian case, this fact ensures the existence of an orthonormal basis diagonalizing $Q$. However, in the Lorentzian case, four different cases can occur, known as {\em Segre types} (see \cite{O'N}, p.261, and for example \cite{Cjgp} for the three-dimensional case). With regard to 
nontrivial three-dimensional paracontact Ricci solitons, as classified in the above Theorem~\ref{thkmu}, with respect to $\{\xi,e,\varphi e\}$, the Ricci operator $Q$ is explicitly given by
$$
Q =\left(\begin{array}{ccc}
-2  & 0   & 0 \\[4pt]
0   & -2\varepsilon a_1-2 & 2a_1 \\[4pt]
0   & -2a_1 &  2\varepsilon a_1-2
\end{array}\right),
$$
which, by a standard calculation, yields to the following result.

\begin{corollary}
If $(M,\varphi,\xi,\eta,g)$ is a  three-dimensional nontrivial paracontact Ricci soliton, then the Ricci eigenvalues are all equal to $-2$, and the corresponding eigenspace is two-dimensional. Therefore, the Ricci operator is of degenerate Segre type $\{(2,1)\}$.
\end{corollary}

\section{Nontrivial paracontact Ricci solitons}
\setcounter{equation}{0}

\subsection{Homogeneous examples}

As proved in \cite{CIll}, a (simply connected, complete) homogeneous paracontact metric three-manifold is isometric to a Lie group $G$ equipped with a left-invariant paracontact metric structure $(\varphi,\xi,\eta,g)$. 
Then, denoting by $\mathfrak{g}$ the Lie algebra of $G$, we have that $\xi \in \mathfrak{g}$, $\eta$  is a $1$-form over $\mathfrak{g}$ and Ker$\eta \subset \mathfrak{g}$. Moreover, starting from a $\varphi$-basis of tangent vectors at the base point of $G$, by left translations one builds a $\varphi$-basis $\{\xi, e, \varphi e \}$ of the Lie algebra $\mathfrak g$.

Suppose now that the left-invariant paracontact metric structure $(\varphi,\xi,\eta,g)$ is a nontrivial Ricci soliton. Then, with respect to the $\varphi$-basis $\{\xi, e, \varphi e \}$ of the Lie algebra $\mathfrak g$, equations \eqref{2h2},\eqref{3h2} and \eqref{RSpc} hold for some real constant $a_1 \neq 0,a_3,a_4,a_5$. The second equation in \eqref{RSpc} then implies $a_3=\varepsilon$, while the last equation in \eqref{3h2} yields $a_5=-\varepsilon a_4$.  Replacing in the expression of $B$, we find $B=2a_3=2\varepsilon$. So, by the last equation of \eqref{RSpc}, we must have 
$$-6=r=-2-2B=-2-4\varepsilon,$$
which yields $\varepsilon=1$. System \eqref{3h2} then reduces to $a_4=a_5=0$. Therefore, \eqref{2h2} now becomes
\begin{equation}\label{leftinv}
	\begin{array}{l}   [\xi,e]= -a_1 e +\left(2-a_1 \right) \varphi e, \quad  
	     [\xi,\varphi e]= \left(2+a_1 \right) e +a_1 \varphi e, \quad  
	     [e,\varphi e]=-2 \xi,
\end{array}\end{equation}
and the Reeb vector field of this paracontact metric structure satisfies \eqref{solit}.
 
The above Lie algebra \eqref{leftinv} is not solvable, as $[\mathfrak g ,\mathfrak g]=\mathfrak g$. Indeed, comparing \eqref{leftinv} with the classification of left-invariant paracontact metric structures obtained in \cite{CIll} (see also \cite{CPA}), we see that \eqref{leftinv} corresponds to the Lie algebra $\mathfrak{sl}(2,\mathbb R)$ of the universal covering of $SL(2,\mathbb R)$. In this way, we proved the following result.

\begin{theorem} \label{teomo}
A homogeneous paracontact metric three-manifold $(M,\varphi,\xi,\eta,g)$ is a nontrivial paracontact Ricci soliton if and only if $M$ is locally isometric to $SL(2,\mathbb R)$, equipped with the left-invariant paracontact metric structure described in \eqref{leftinv}.  
\end{theorem}

\subsection{Explicit paracontact metric structures and some inhomogeneous examples}

We begin with the following result.

\begin{proposition} \label{tb}
Any three-dimensional paracontact metric structure $(\varphi, \xi, \eta, g)$, in terms of local Darboux cordinates $(x,y,z)$, is explicitly described by 
$$	\xi=2 \partial_z, \quad \eta=\frac{1}{2}(dz-ydx), $$
\begin{equation*}  
	g=\frac{1}{4} \left(
\begin{array}{ccc}
	a & b & -y  \\[4pt]
	b & c & 0   \\[4pt]
	-y& 0 & 1 
\end{array}
	\right),
\qquad 
\varphi=\left(
\begin{array}{ccc}
	-b      & -c  & 0  \\[4pt]
	(a-y^2) &  b  & 0   \\[4pt]
	-by     & -cy & 0 
\end{array}
	\right)
\end{equation*}
%
for some smooth functions $a,b,c$, satisfying $ac-b^2-cy^2=-1$. In particular,

\smallskip\noindent
i) the structure is paraSasakian if and only if the functions $a,b,c$ do not depend on $z$, 

\smallskip\noindent
ii) $h^2=0$ (equivalently, ${\rm tr} h^2=0$) if and only if $b_z^2-a_zc_z=0$.

\end{proposition}

\begin{proof}
By the classic Theorem of Darboux (see \cite{B}, p.24)  a contact form $\eta$ and {its} Reeb vector field $\xi$ are explicitly given by $\eta=\frac{1}{2}(dz-ydx)$ and $\xi=2 \partial_z$ in terms of Darboux coordinates $(x,y,z)$.
 
Next, compatibility condition $g(\xi,X)=\eta(X)$ yields the form of the last row and column of the matrix representing $g$ with respect to $\{{\partial_x},{\partial_y},{\partial_z}\}$. We then observe that, with respect to this basis, one has
\begin{equation*}
d\eta=\frac 14 \left(
\begin{array}{ccc}
	0    & 1  & 0  \\[4pt]
	-1 &  0   & 0   \\[4pt]
	0    &  0   & 0 
\end{array}
	\right),
\qquad   
	g(\cdot,\varphi)=-\frac{ac-b^2-cy^2}{4} \left(
\begin{array}{ccc}
	0  & 1 & 0  \\[4pt]
	-1 & 0 & 0   \\[4pt]
	0  & 0 & 0 
\end{array}
	\right)
\end{equation*}
So, $g$ is an associated metric, that is, $d\eta=g(\cdot,\varphi)$, if and only if $ac-b^2-cy^2=-1$.

\noindent
Finally, with respect to $\{{\partial_x},{\partial_y},{\partial_z}\}$, tensors $h=\frac 12 \mathcal L_{\xi} \varphi$  and $h^2$ are explicitly given by
\begin{equation}  \label{hmatrix}
	h=\left(
\begin{array}{ccc}
	-b_z  & -c_z  & 0  \\[4pt]
	a_z   & b_z   & 0   \\[4pt]
	-b_z y& -c_zy & 0 
\end{array}
	\right),
\qquad 
h^2=\left(
\begin{array}{ccc}
	 b_z^2-a_z c_z     & 0              & 0  \\[4pt]
	  0                & b_z^2-a_z c_z  & 0   \\[4pt]
	 (b_z^2-a_z c_z)y  & 0              & 0 
\end{array}
	\right).
\end{equation}
A three-dimensional paracontact metric manifold is paraSasakian if and only if $h=0$ \cite[Theorem~2.2]{CIll}. So, by \eqref{hmatrix}, the structure is paraSasakian if and only if $a_z=b_z=c_z=0$, while $h^2=0$ if and only if $b_z^2-a_zc_z=0$.
\end{proof}

\noindent
The above result emphasizes the fact that differently from the contact metric case,  for paracontact metric structures condition $h^2=0$ does not imply $h=0$.

\bigskip \noindent
{\bf Example.} We consider $M=\mathbb R^3(x,y,z)$, equipped with the paracontact metric structure $(\varphi, \xi, \eta, g)$ defined by
\begin{eqnarray} \label{xieta}
		\xi=2 \partial_z, \quad \eta=\frac{1}{2}(dz-ydx),
\end{eqnarray}
%
%
\begin{equation}  \label{gvarphi}
	g=\frac{1}{4} \left(
\begin{array}{ccc}
	F & 1 & -y  \\[4pt]
	1 & 0 & 0   \\[4pt]
	-y& 0 & 1 
\end{array}
	\right),
\qquad 
\varphi=\left(
\begin{array}{ccc}
	-1     & 0  & 0  \\[4pt]
	F-y^2 & 1  & 0   \\[4pt]
	-y     & 0  & 0 
\end{array}
	\right),
\end{equation}
where 
$$F=F(x,y,z)=f(x)+\alpha e^{2z} + \beta y+ \gamma ,$$ 
for a  smooth function $f(x)$ and some real constant $\alpha \neq 0, \beta, \gamma$. We now prove the following result.


\begin{theorem}\label{nothomRS}
Let $(\varphi, \xi,\eta, g)$ be the paracontact metric structure described by \eqref{xieta} and \eqref{gvarphi}. Then, $(\mathbb R^3, \varphi, \xi,\eta, g)$ is a paracontact Ricci soliton. Moreover, for any $\beta\neq 0$, such a paracontact metric structure is not locally homogeneous.
\end{theorem}

\begin{proof} 
The paracontact metric structure defined by \eqref{xieta} and \eqref{gvarphi} is of the type described in Proposition~\ref{tb}, with $a=F$, $b=1$ and $c=0$. This structure is not paraSasakian, because $a_z=F_z =2\alpha e^{2z}\neq 0$. On the other hand, since $b,c$ are constant, one concludes at once that $b_z^2-a_zc_z=0$. Therefore, $h^2=0 \neq h$. 

We now determine a global $\varphi$- basis $(\xi, E, \varphi E)$ on $M$, taking
$$E:=\frac{1}{\sqrt{2}} \big(4 \partial_x + (2y^2 - 2F +1) \partial_y +4y\partial_z \big)$$
and so,
$$\varphi E=\frac{1}{\sqrt{2}} \big(-4 \partial_x + (1-2y^2+ 2F ) \partial_y -4y\partial_z \big).$$ 
By a standard calculation, we then get
\begin{equation} \label{lb} 
\left\{
 \begin{array}{l}
    [\xi,E]=- 4 \alpha e^{2z}(E+\varphi E),  \\[4pt]
     [\xi,\varphi E]=4 \alpha e^{2z}(E+\varphi E), \\[4pt]
 [E,\varphi E]=-2\xi+\sqrt{2}(\beta-2y)(E+\varphi E). 
 \end{array}	
\right.
\end{equation}

\noindent
We can now compare \eqref{lb} with \eqref{2pc}, obtaining $a_1=a_2=4\alpha e^{2z}$, $a_3=-1$ and $a_4=-a_5=\sqrt{2}(\beta-2y)$. It is then easy to check that the conditions in \eqref{3h2} and \eqref{RSpc} are satisfied. Hence, by Theorem~\ref{thRS}, we conclude that $(M,\varphi, \xi, \eta, g)$ is a paracontact Ricci soliton.

\smallskip
Suppose now that $(\mathbb R^3, \varphi, \xi,\eta, g)$ is a locally homogeneous paracontact metric manifold.  Then, by \cite{CIll} (see also \cite{CPA}), this manifold is locally isometric to some Lie group, equipped with a left-invariant paracontact metric structure.

\noindent
In particular, then there exist (at least, locally) a left-invariant orthonormal $\varphi$-basis $\{\xi, e, \varphi e \}$ of the corresponding Lie algebra, that is, a new basis $\{e, \varphi e \}$ of ker$\eta$, such that
\begin{equation*}  
    [\xi,e]=\lambda_1 \xi+\lambda_2 e+\lambda_3 \varphi e,  \quad
     [\xi,\varphi e]=\mu_1 \xi+\mu_2 e+\mu_3 \varphi e,  \quad 
		[e,\varphi e]=\sigma_1 \xi+\sigma_2 e+\sigma_3 \varphi e, 
\end{equation*}
for some real constants $\lambda_i, \mu_i, \sigma_i$, $i=1,2,3$. Hence, there must exist two smooth functions $p=p(x,y,z)$ and $q=q(x,y,z)$, such that $p^2-q^2=1$ and
$$
	e=p E+ q\varphi E, \qquad \varphi e= q E+p \varphi e.
$$
Conversely, we then get
$$
E=p e-q\varphi e {\quad \rm and \quad} \varphi E= -q e+ p \varphi e.
$$
Therefore, a standard calculation yields 
\begin{equation*}  
\left\{
 \begin{array}{l}
[\xi,e]=\big(2pp_z - 2qq_z - 4\alpha e^{2z} (p-q)^2 \big) e+\big(2pq_z - 2qp_z - 4\alpha e^{2z} (p-q)^2 \big) \varphi e,  \\[4pt]
[\xi,\varphi e]=\big(2pq_z - 2qp_z + 4\alpha e^{2z} (p-q)^2 \big) e+\big(2pp_z - 2qq_z + 4\alpha e^{2z} (p-q)^2 \big) \varphi e.
 \end{array}	
\right.
\end{equation*}
Note that $p\neq q$, as $p^2-q^2=1$. Requiring that the coefficients in the above Lie brackets are costant, we then  easily find
\begin{eqnarray}
	p=\frac{1}{2}\Big(\frac{e^z}{C}+\frac{C}{e^z}\Big) \quad \text{and} \quad	q=\frac{1}{2}\Big(\frac{e^z}{C}-\frac{C}{e^z}\Big),
\end{eqnarray}
for some real costant $C\neq 0$. Consequently, we have

\begin{equation*}  
\left\{
 \begin{array}{l}
[\xi,e]=-(4\alpha C^{2}) e+\big(2 - 4\alpha C^{2}) \varphi e,  \\[4pt]
[\xi,\varphi e]=(2 + 4\alpha C^{2} ) e+ 4\alpha C^{2} \varphi e,   \\[4pt]
[e,\varphi e]= -2\xi + \sqrt{2}\beta C e^{-z} (e+ \varphi e),
 \end{array}	
\right.
\end{equation*}

\noindent
where $\alpha, C \neq 0$ and $\beta$ are real constants. Therefore, $(\mathbb R^3, \varphi, \xi,\eta, g)$ is locally homogeneous if and only if $\beta=0$. It is easily seen, from the Lie brackets above, that for $\beta=0$ we get exactly the Lie group $SL(2,\mathbb R)$ of Theorem~\ref{teomo}. On the other hand, whenever $\beta \neq 0$, we described a paracontact Ricci soliton which is not locally homogeneous. We also checked using {\em Maple16$^\copyright$} that homogeneous Lorentzian structures \cite{Cjgp} exist for the metric $g$ if and only if $\beta=0$, so confirming that when $\beta \neq 0$, this example is indeed not locally homogeneous.

\end{proof}

Contact $(\kappa,\mu)$-spaces (with $\kappa,\mu$ constant) have been completely classified in \cite{Bo}, showing that the non-Sasakian ones are isometric to some Lie groups equipped with a left-invariant contact metric structure. On the other hand, as a consequence of the above Theorems~\ref{thkmu}  and \ref{nothomRS}, there exist three-dimensional paracontact $(\kappa,\mu)$-spaces, which are neither paraSasakian nor isometric to Lie groups.

\section{3D Normal almost paracontact Ricci solitons}
\setcounter{equation}{0}

We now consider normal almost paracontact metric three-manifolds, obtaining for them a description similar to the one given in Section~3 for the paracontact metric case. As we already mentioned in Section~2, normal almost paracontact metric manifolds satisfy conditions listed in Proposition~\ref{first}. Consequently, they are locally described by equations~\eqref{2},\eqref{3}, for some smooth functions, $a_i,b_j$, with respect to a local $\varphi$-basis $\{\xi,e,\varphi e\}$. Moreover, a three-dimensional almost paracontact metric manifold is normal if and only if $h=0$ \cite{CPA}. Thus, by \eqref{acca} we have at once that $a_1=a_2=0$ and so, the Levi-Civita connection of a normal almost paracontact metric three-manifold is described by \eqref{nabla} taking $a_1=a_2=0$. Observe that from \eqref{trphi} and \eqref{divxi} we now have that $b_1=\frac{1}{2} {\rm div} \xi$ and $b_2 = \frac{1}{2} {\rm tr} \varphi \nabla \xi$ are globally defined on $M$ (see also \cite{BeCr}).

We can now follow the same argument used in Section~3 for paracontact metric manifolds, to describe the cases when $\xi$ is an infinitesimal harmonic transformation and, in particular, when $\xi$ determines a solution to the Ricci soliton equation~\eqref{solit}. Also taking into account \eqref{3} (with $a_1=a_2=0$), we find
\begin{align*}
&(\mathcal L_{\xi} \nabla)(\xi,\xi)=0, \\[2pt]	   
&(\mathcal L_{\xi} \nabla)(e,e)=-\big(\xi(b_1)+2b_1^2\big)\xi+e(b_1)e +\varphi e(b_1) \varphi e, \\[2pt]	   
&(\mathcal L_{\xi} \nabla)(\varphi e ,\varphi e)=\big(\xi(b_1)+2b_1^2\big)\xi+e(b_1) e+ \varphi e(b_1) \varphi e 
\end{align*}
and so,
\begin{align*}
{\rm tr}(\mathcal L_{\xi} \nabla)=& (\mathcal L_{\xi} \nabla)(\xi,\xi)+(\mathcal L_{\xi} \nabla)(e ,e)-(\mathcal L_{\xi} \nabla)(\varphi e ,\varphi e) =-2\big(\xi(b_1)+2b_1^2\big)\xi . 
\end{align*}
Hence, we have the following result.

\begin{proposition}
A three-dimensional normal almost paracontact metric manifold $(M,\varphi,\xi,\eta,g)$ is locally described by 
\begin{equation} \label{2nor}
    [\xi,e]=-b_1 e+(a_3-b_2)\varphi e,  \;
     [\xi,\varphi e]=(a_3-b_2)e-b_1\varphi e, \;
     [e,\varphi e]=2b_2 \xi+a_4e-a_5\varphi e,
\end{equation}
with respect to a local $\varphi$-basis $\{\xi, e, \varphi e\}$, for some smooth functions $b_1,b_2,a_3,a_4,a_5$, satisfying 
\begin{eqnarray} \label{3nor}
\left\{ \begin{array}{l}
				\xi(b_2)+2b_1 b_2=0,	\\[2pt]
				-\xi(a_4)+e(a_3-b_2) + \varphi e(b_1) -b_1a_4+a_5(a_3-b_2)=0,	\\[2pt]
				\xi(a_5)-e(b_1) - \varphi e(a_3-b_2) +b_1a_5-a_4(a_3-b_2)=0.
\end{array}
\right.
\end{eqnarray}
In particular, the Reeb vector field $\xi$ is an infinitesimal harmonic transformation if and only if 
\begin{equation}\label{IHTnor}
\xi(b_1)+2b_1^2=0.  
\end{equation}
\end{proposition}

We now consider the case when $\xi$ is an affine Killing vector field. In such a case, from \eqref{nablapc} (with $a_1=a_2=0$) we find $(\mathcal L_{\xi} \nabla)(\xi,e)=\xi(b_1)e+\xi(b_2)\varphi e=0$ and $(\mathcal L_{\xi} \nabla)(e,e)=0$. Thus,  we get 
\begin{eqnarray*}
\left\{ \begin{array}{l}
\xi(b_1)=\xi(b_2)=0, \\[2pt]	
\xi(b_1)+2b_1^2=0,
\end{array}
\right.
\end{eqnarray*}
%
which easily yield $b_1=0$. So, by Proposition~\ref{p5}, $\xi$ is Killing and we proved  the following rigidity result.

\begin{theorem}
If the Reeb vector field $\xi$ of a three-dimensional normal almost paracontact metric manifold $(M,\varphi,\xi,\eta,g)$ is affine Killing, then $\xi$ is Killing. In particular, in such a case,
\begin{enumerate}
\item if $b_2=\frac{1}{2} {\rm tr}(\varphi\nabla \xi) \neq 0$, then $(M,\varphi,\xi,\eta,g)$ is quasi-para-Sasakian;
\vspace{2pt}\item if $b_2=\frac{1}{2} {\rm tr}(\varphi\nabla \xi)=0$, then $(M,\varphi,\xi,\eta,g)$ is para-cosymplectic.
\end{enumerate}
\end{theorem}

\noindent
Next, with regard to the curvature and the Ricci tensor, using \eqref{nabla} with $a_1=a_2=0$, by standard calculations we find 
\begin{align*}
&R(\xi,e)\xi= -\big( \xi(b_1)+b_1^2 +b_2 ^2 \big)e , \\[2pt]	
&R(\xi, \varphi e)\xi= -\big( \xi(b_1)+b_1^2 +b_2 ^2 \big)\varphi e , \\[2pt]
&R(e, \varphi e)\xi= \big(\varphi e (b_1)-e(b_2) \big)e +\big(\varphi e (b_2)-e(b_1)\big) \varphi e, \\[2pt]
&R(e, \varphi e)e= -\big(\varphi e (b_1)-e(b_2) \big)\xi+ \big(\varphi e (a_4)-e(a_5)+a_4^2-a_5^2-b_1^2+b_2^2+2b_2a_3\big) \varphi e 	
\end{align*}
and so, with respect to $\{\xi, e, \varphi e \}$, the Ricci tensor $\varrho$ is completely described by
\begin{equation}\label{rhonor}
\varrho =\left(\begin{array}{ccc}
-2 \tilde{A}& \varphi e (b_2)-e(b_1) & e (b_2)-\varphi e(b_1) \\[4pt]
\varphi e (b_2)-e(b_1) & \tilde B-\tilde A & 0 \\[4pt]
 e (b_2)-\varphi e(b_1) & 0 & -\tilde B+\tilde  A
\end{array}\right),  
\end{equation}
with
\begin{equation}\label{ABnor}
\tilde A:=\xi(b_1)+b_1^2 +b_2 ^2, \qquad
\tilde B:= \varphi e (a_4)-e(a_5)+a_4^2-a_5^2-b_1^2+b_2^2+2b_2a_3. 
\end{equation}

It is worthwhile to remark that, by \eqref{rhonor}, contrarily to the paracontact metric case, even when the Reeb vector field $\xi$ of a normal almost paracontact metric structure is an infinitesimal harmonic transformation, $\xi$ needs not to be a Ricci eigenvector. 

From the above formulas we have that $R(\xi,e,\xi,e)=- \tilde A$. On the other hand, taking $X=Z=\xi$ and $Y=V=e$ in \eqref{curv3d}, we get $R(\xi,e,\xi,e)=\varrho(e,e)-2 \tilde A- \frac{r}{2}$ and so, $\varrho(e,e)=\tilde A + \frac{r}{2}$. Comparing this equation with \eqref{rhonor}, we then find $\tilde B=2 \tilde A+  \frac{r}{2}$. Consequently, \eqref{rhonor} becomes
\begin{equation}\label{rhonor*}
\varrho =\left(\begin{array}{ccc}
-2 \tilde{A}& \varphi e (b_2)-e(b_1) & e (b_2)-\varphi e(b_1) \\[4pt]
\varphi e (b_2)-e(b_1) & \tilde A+  \frac{r}{2} & 0 \\[4pt]
 e (b_2)-\varphi e(b_1) & 0 & -\tilde A-  \frac{r}{2}
\end{array}\right),  
\end{equation}
We now use \eqref{nablapc} with $a_1=a_2=0$ to describe the Lie derivative of $g$ with respect to $\{\xi, e, \varphi e \}$ and we find 
$$
\mathcal L _{\xi} g =\left(\begin{array}{ccc}
0 & 0 & 0 \\[2pt]
0 & 2 b_1 & 0 \\[2pt]
0 & 0 & -2 b_1
\end{array}\right).
$$
Henceforth, $\xi$ satisfies equation \eqref{solit} if and only if the following system of differential equations holds:
\begin{equation} \label{RSnor1}
\left\{
\begin{array}{l}
 \lambda+2 \tilde A=0, \\[2pt]
2b_1+\tilde A +\frac{r}{2}- \lambda=0, \\[2pt]
\varphi e (b_2)-e(b_1)=0 ,\\[2pt]
e (b_2)-\varphi e(b_1)=0. 
\end{array}
\right.
\end{equation}
%
%
%
By  \eqref{IHTnor}, the first equation in \eqref{RSnor1} and the definition of $\tilde A$, we have
$$-2b_1^2= \xi(b_1)=-\frac{\lambda}{2}-b_1^2-b_2^2 $$
and so, 
\begin{equation}\label{last*}
2(b_1^2-b_2^2)=\lambda .
\end{equation}
Deriving \eqref{last*} with respect to $\xi$ and substituting $\xi(b_2)$ from the first equation in \eqref{3nor},
we then get $2b_1 (b_2^2-b_1^2)=0$, that is, again by \eqref{last*},
$$\lambda b_1=0.$$
If $\lambda\neq 0$, then the equation above implies that $b_1=0$. Thus, $\mathcal L _\xi g=0$, that is, $\xi$ is Killing and so, the Ricci soliton is trivial. Moreover, by  \eqref{rhonor*}, when $b_1=0$ we easily conclude that the three-dimensional normal almost paracontact metric manifold is Einstein and  so, of constant curvature $k=-b_2^2 \leq 0$. 

\noindent
On the other hand, if $\lambda=0$, then \eqref{RSnor1} and \eqref{last*} yield
\begin{equation}\label{RSnor0}
\left\{
\begin{array}{l}
b_2=\varepsilon b_1 =\pm b_1, \\[2pt]
\xi(b_1)=-2b_1^2, \\[2pt]
r=-4b_1, \\[2pt]
(e-\varepsilon \varphi e) (b_1)= 0. 
\end{array}
\right.
\end{equation}
Summarizing, we proved the following result.
\begin{theorem}\label{thRSnor*}
Let $(M,\varphi,\xi,\eta,g)$ a normal almost paracontact metric three-manifold, as locally described by equations \eqref{2nor}-\eqref{IHTnor}. 

\smallskip\noindent
1) If $(M,\varphi,\xi,\eta,g)$ is an unsteady almost paracontact Ricci soliton, then $\xi$ is Killing and $(M,g)$ has constant sectional curvature $k=-b_2^2 < 0$. 

\smallskip\noindent
2) If $(M,\varphi,\xi,\eta,g)$ is a steady almost paracontact Ricci soliton, then  \eqref{RSnor0} holds.
\end{theorem}

\noindent
By the above Theorem~\ref{thRSnor*}, a normal almost paracontact unsteady Ricci soliton is necessarily trivial, and of negative constant sectional curvature.  Moreover, in the case $2)$ of the above Theorem, if $\xi(r)=0$ then by \eqref{RSnor0} we get $-2b_1^2=\xi(b_1)=0$. So, $b_2=b_1=0$, that is, $(M,g)$ is flat. Therefore, we have the following result.

\begin{corollary}
A normal almost paracontact metric three-manifold $(M,\varphi,\xi,\eta,g)$  with $\xi(r)=0$ (in particular, having constant scalar curvature, or being homogeneous) is a Ricci soliton if and only if $(M,g)$ is of nonpositive constant sectional curvature.
\end{corollary}

The above Theorem~\ref{thRSnor*} clarifies the structure of three-dimensional normal almost paracontact Ricci solitons, and permits us to correct Theorem~3.1 (and subsequent Proposition~3.6) in \cite{BeCr}. In fact, Theorem~3.1 of 
\cite{BeCr} states that if $(M,\varphi,\xi,\eta,g)$ is an almost paracontact Ricci soliton, then $\xi$ is a conformal Killing vector field, that is, one would have
$$\mathcal L _\xi g = f \cdot g,$$
for some smooth function $f$. However, whether the manifold is a Ricci soliton or not, since $g(\xi,\xi)=1$, from the above equation one would have at once
$$f= f \cdot g(\xi,\xi)= (\mathcal L _\xi g)(\xi,\xi)= \xi (g(\xi,\xi))-2g([\xi,\xi],\xi) =0,$$
so that $\xi$ would be Killing. In particular, if $(M,\varphi,\xi,\eta,g)$ were a Ricci soliton, then it would be an Einstein manifold (and so, being three-dimensional, of constant sectional curvature). The same conclusion also follows by the fact that substituting $\mathcal L _\xi g = f \cdot g$ into the Ricci soliton equation \eqref{solit}, one would get that the Ricci tensor $\varrho$ satisfies $\varrho =(\lambda -f) \cdot g$, and it is well known that this yields that $f$ is necessarily a constant and so, the manifold is Einstein. 

Thus, we can correct Theorem~3.1 in \cite{BeCr} by the following rigidity result.

\begin{theorem}\label{thRSBeCr}
Let $(M,\varphi,\xi,\eta,g)$ be a normal almost paracontact metric three-manifold. If $\xi$ is a conformal Killing vector field, then $\xi$ is Killing. In particular, if $(M,\varphi,\xi,\eta,g)$ is also an almost paracontact  Ricci soliton, then it is trivial, that is, $(M,g)$ has (nonpositive) constant sectional curvature.
\end{theorem}


\end{document}